\documentclass[11pt]{amsart}
\usepackage{amsmath,amsthm, amscd, amssymb, amsfonts, mathtools,color}
\usepackage[all]{xy}
\usepackage{mathrsfs}
\usepackage{enumitem}
\usepackage[T2A,T1]{fontenc}
\usepackage{multicol}

\usepackage{hyperref}

\allowdisplaybreaks

\usepackage[ansinew]{inputenc}
\usepackage{graphicx,fancyhdr}

\newcommand{\deriv}{\mathfrak d}
\newcommand{\derv}{\mathfrak D}

\newcommand{\nucleo}{\mathbf R}
\newcommand{\Nuc}{\mathcal O}

\newcommand{\Gb}{\mathbf G}

\newcommand{\verma}{M}

\numberwithin{equation}{section}
\newtheorem{theorem}{Theorem}[section]
\newtheorem{lemma}[theorem]{Lemma}
\newtheorem{coro}[theorem]{Corollary}

\newtheorem{prop}[theorem]{Proposition}

\theoremstyle{definition}
\newtheorem{definition}[theorem]{Definition}

\theoremstyle{remark}
\newtheorem{remark}[theorem]{Remark}

\newcommand{\pf}{\begin{proof}}
\newcommand{\epf}{\end{proof}}

\newcommand{\spl}{\mathfrak{sl}}

\newcommand{\ku}{ \Bbbk}

\newcommand{\kut}{ \ku^{\times}}

\newcommand{\I}{\mathbb I}

\newcommand{\N}{\mathbb N}

\newcommand{\Z}{\mathbb Z}

\newcommand{\tx}{\mathtt{x}}

\newcommand{\cO}{\mathcal{O}}

\newcommand{\cF}{\mathcal{F}}

\newcommand{\cC}{\mathcal{C}}

\newcommand{\D}{\mathcal{D}}

\newcommand{\cH}{\mathcal{H}}
\newcommand{\cI}{\mathcal{I}}

\newcommand{\Pc}{{\mathcal P}}

\newcommand{\Ss}{{\mathcal S}}

\newcommand{\kh}{\mathfrak h}

\newcommand{\ad}{\operatorname{ad}}

\newcommand{\Frac}{\operatorname{Frac}}

\newcommand{\car}{\operatorname{char}}

\newcommand{\End}{\operatorname{End}}
\newcommand{\id}{\operatorname{id}}
\newcommand{\gr}{\operatorname{gr}}

\newcommand{\Hom}{\operatorname{Hom}}

\newcommand{\Tr}{\operatorname{Tr}}
\newcommand{\Irr}{\operatorname{irrep}}
\newcommand{\IRR}{\operatorname{Irrep}}
\newcommand{\Ind}{\operatorname{Ind}}
\newcommand{\Res}{\operatorname{Res}}

\newcommand{\toba}{\mathscr{B}}
\newcommand{\Ds}{\mathscr{D}}
\newcommand{\ot}{\otimes}

\newcommand{\yd}[1]{{}^{ #1 }_{ #1 }\mathcal{YD}}
\newcommand{\dy}[1]{\mathcal{YD}^{ #1 }_{ #1 }}
\newcommand{\lmod}[1]{{}_{ #1 }\mathcal{M}}

\newcounter{tabla}\stepcounter{tabla}

\begin{document}
\noindent
\title[On the double of the Jordan plane]
{On the double of the Jordan plane}

\author[Andruskiewitsch]{Nicol\'as Andruskiewitsch}
\address[N.Andruskiewitsch]{Facultad de Matem\'atica, Astronom\'ia y F\'isica,
Universidad Nacional de C\'ordoba. CIEM -- CONICET. 
Medina Allende s/n (5000) Ciudad Universitaria, C\'ordoba, Argentina}
\email{nicolas.andruskiewitsch@unc.edu.ar}
\author[Dumas]{Fran\c cois Dumas}
\address[F.~Dumas]{Universit\'e Clermont Auvergne, CNRS, LMBP, F-63000 Clermont-Ferrand, France}
\email{Francois.Dumas@uca.fr}

\author[Pe\~na Pollastri]{H\'ector Mart\'\i n Pe\~na Pollastri}
\address[H.~Pe\~na Pollastri]{Facultad de Matem\'atica, Astronom\'ia y F\'isica, Universidad Nacional de C\'ordoba. CIEM -- CONICET. 
Medina Allende s/n (5000) Ciudad Universitaria, C\'ordoba, Argentina}
\email{hector.pena.pollastri@unc.edu.ar}

\thanks{}

\date{\today} 


\begin{abstract}
We compute the simple finite-dimensional modules and the center of the Drinfeld double of the Jordan plane 
introduced in \texttt{arXiv:2002.02514} assuming that the  characteristic is zero.
\end{abstract} 
\maketitle

\section{Introduction}
Let $\ku$ be a field.
The well-known Jordan plane is the quadratic algebra $J = \ku\langle x,y \vert xy - yx - \frac{1}{2}x^2\rangle$;
it bears a structure of braided Hopf algebra where $x$ and $y$ are primitive \cite{G}.
When $\ku$ has characteristic 0 $J$ is indeed a Nichols algebra (any primitive element belongs to $V = \ku x \oplus \ku y$)
but if $\car \ku = p >0$, then $J$ covers the  Nichols algebra $\toba(V)$ 
which has now finite dimension \cite{clw}. In \cite{ap} $\toba(V)$  was called the  restricted Jordan plane and,
assuming $p > 2$, the Drinfeld double $D(\cH)$ of the bosonization $\cH = \toba(V) \# \ku C_p$ was studied. 
It was shown that $D(\cH)$ fits into an exact sequence of Hopf algebras 
$\xymatrix{ \nucleo \ar@{^{(}->}[r]  & D(\cH) \ar@{->>}[r]  & \mathfrak u(\spl_2(\ku))}$
where $\nucleo$ is local commutative  and $\mathfrak u(\spl_2(\ku))$
is the restricted enveloping algebra of $\spl_2(\ku)$. 
Hence the simple $D(\cH)$-modules are the same as those of $\mathfrak u(\spl_2(\ku))$ \cite[1.11]{ap}. 

\medbreak 
In \cite{ap} a Hopf algebra $\D$ covering $D(\cH)$ was defined, see Section \ref{sec:double};
$\D$ can be thought of as the Drinfeld double of the bosonization 
$J \# \ku \Z$.
Now $\D$ fits into an exact sequence of Hopf algebras 
$\xymatrix{ \cO(\Gb)  \ar@{^{(}->}[r]  & \D \ar^{\pi \qquad}@{->>}[r]  & U(\spl_2(\ku)) }$
where $\Gb = (\Gb_a \times \Gb_a) \rtimes \Gb_m$, $\cO(\Gb)$ is the algebra of regular functions and 
$U(\spl_2(\ku))$ is the enveloping algebra.

\medbreak
Both the definition of $\D$ and the exact sequence are still valid in characteristic 0, what we assume from now on.
In this paper we offer two results on  the structure of $\D$. First
we classify  in  Section \ref{sec:simples-D}  the finite-dimensional irreducible representations of $\D$; the outcome ressembles the case of $ D(\cH)$, 
analogy supported by Lemma \ref{lema:quotient-nilpotent}. Concretely, we prove:

\medbreak
\textbf{Theorem  \ref{thm:simple-D}. }
\emph{There is a bijection $\Irr \D \simeq \Irr \spl_2(\ku)$ induced by the morphism $\pi : \D \to U(\spl_2(\ku))$. }

\medbreak
Second  we consider in Section \ref{sec:localization} the localization $\D'$ of $\D$ at the powers of $x$ and $q:=ux + 2(1+g)$.
We show in Theorem \ref{thm:D'}  that $\D'$ is isomorphic to a localization of the  Weyl algebra 
$A_2(S)$, where $S:=\ku[z^{\pm 1},z']$ with $z$ and $z'$ algebraically independent.
This result allows us to compute the center of $\D$, that turns out to be a Kleinian singularity of type $A_1$,  see Theorem
\ref{prop:center-D}. Also by Theorem \ref{thm:D'}, $\D$ satisfies the Gelfand-Kirillov property,  see \S \ref{subsec:GK}.

\subsection*{Conventions}
If $\ell < n \in\N_0$, we set $\I_{\ell, n}=\{\ell, \ell +1,\dots,n\}$, $\I_n = \I_{1, n}$. 
If $Y$ is a subobject of an object $X$ in a category $\cC$, then we write $Y\leq X$.

In the rest of the paper $\ku$ is an algebraically closed field of characteristic 0.
The subspace of a vector space $V$ generated by $S\subset V$ is denoted by $\ku S$.
Let $A$ be an algebra and $a_1,\dots,a_n\in A$, $n\in\N$. Let $\mathcal Z(A)$ denote the center of $A$.
We denote by $\ku\langle a_1,\dots,a_n\rangle$ the subalgebra generated by $a_1,\dots,a_n$.
An element $x\in A$ is normal if $Ax = xA$. 
If $A = \bigoplus_{n\in \Z} A^n$ is graded and $T\subseteq A$ is a subspace, then $T^n \coloneqq T\cap A^n$.
If $M$ is an $A$-module and $m_1,\dots,m_n\in A$, $n\in\N$, then we denote by $\langle m_1,\dots,m_n\rangle$ the submodule generated by $m_1,\dots,m_n$.

Let $L$ be a Hopf algebra. The kernel of the counit $\varepsilon$ is denoted $L^+$, the antipode (always assumed bijective) by $\Ss$, the
space of primitive elements by $\Pc(L)$ and the group of group-likes by $G(L)$.
The space of $(g,h$)-primitives is $\Pc_{g,h}(L) =\{x \in L: \Delta(x) = x\otimes h + g \otimes x\}$ where $g,h \in G(L)$.
The category of left-left, respectively right-right, Yetter-Drinfeld modules over $L$ is denoted by $\yd{L}$, respectively $\dy{L}$.
We refer to \cite{Rad-libro} for unexplained terminology on Hopf algebras.

\section{The double of the Jordan plane}\label{sec:double}

\begin{definition} \label{def:double-jordan} \cite[2.3]{ap}
The Hopf algebra $\D$ is presented by generators $u$, $v$, $\zeta$, $g^{\pm 1}$, $x$, $y$ and relations
\begin{align}
\label{eq:def-jordan1}
&\begin{aligned}
g^{\pm 1} g^{\mp 1}&= 1, &\zeta  g &= g \zeta, 
\end{aligned}
\\ \label{eq:def-jordan2}
&\begin{aligned}
gx &= xg, & gy &= yg +xg, & \zeta  y &= y \zeta  + y, &  \zeta  x&=x \zeta  + x, 
\\
u g &= gu, & v g &=gv  + gu, & v\zeta  &= \zeta v +v, &  u\zeta  &= \zeta u +u,
\end{aligned}
\\
\label{eq:def-jordan3}
&\begin{aligned}
yx &= xy - \frac{1}{2} x^2, & vu &= uv - \frac{1}{2}u^2,
\\
u x &=x u , &   v x &= x v  + (1-g) + xu, 
\\
u y &=yu  +(1 - g), &  v y &=y v -g \zeta  + y u. 
\end{aligned}
\end{align}

The comultiplication is defined by $g\in G(\D)$, $u,\zeta\in\Pc(\D)$, $x,y\in\Pc_{g,1}(\D)$,
\begin{align*}
\Delta( v) &=  v\ot 1 + 1\ot v +\zeta\ot u.
\end{align*}
\end{definition}

We list some basic properties of $\D$, cf. \cite{ap} for details.

\bigbreak
\noindent $\bullet$  The following set  is a PBW-basis of $\D$:
\begin{align*}
B &= \{x^n\,y^r\,g^m\,\zeta ^k\,u^i\,v^j: i,j,k,n,r\in\N_0, \quad m\in\Z\}.
\end{align*}

\bigbreak
\noindent $\bullet$  $\D = \oplus_{n\in\Z}\D^{[n]}$ is $\Z$-graded by 
\begin{align}\label{eq:D-grading}
\deg x =\deg y = -1, && \deg u =\deg v = 1, && \deg g = \deg \zeta = 0.
\end{align}

\bigbreak
\noindent $\bullet$  Let $\Gamma$ be the infinite cyclic group with generator $g$ written multiplicatively and 
let $\mathfrak{h}$ be the one dimensional  Lie algebra. 
The subalgebra  $\D^0 = \ku \langle g^{\pm1}, \zeta \rangle$ is a Hopf subalgebra isomorphic to $ \ku \Gamma\otimes U(\mathfrak h)$.

\bigbreak
\noindent $\bullet$   The subalgebra  $\D^{< 0}= \ku \langle x, y \rangle$ is isomorphic to the Jordan plane $J$.
It is a Hopf algebra in $\yd{\ku \Gamma}$ and the bosonization $\D^{< 0}\# \ku \Gamma \simeq \ku \langle g^{\pm1}, x, y \rangle$.

\bigbreak
\noindent $\bullet$   The subalgebra  $\D^{> 0} = \ku \langle u, v \rangle$ is   isomorphic to $J$ as algebra via $u\mapsto x$ and $v\mapsto y$,
but with a different comultiplication; actually $\D^{> 0}$ is the graded dual of $\D^{< 0}$. 
Then $\D^{> 0}$ is a Hopf algebra in $\dy{U(\mathfrak h)}$ and  $U(\mathfrak h) \# \D^{> 0} \simeq \ku \langle \zeta, u, v \rangle$.

\bigbreak
\noindent $\bullet$  The  subalgebras $\D^{> 0}$,  $\D^{0}$ and $\D^{<0}$ are graded and satisfy
\begin{enumerate}[leftmargin=*,label=\rm{(\alph*)}]
\smallbreak\item $\D^{> 0}\subseteq\oplus_{n\in\N_0} \D^{[n]}$, $\D^{< 0}\subseteq\oplus_{n\in \N_0} \D^{[-n]}$.

\smallbreak
\item $(\D^{> 0})^{[0]} = \ku = (\D^{< 0})^{[0]}$.

\smallbreak
\item $\D^{\geq0} := \D^0 \D^{> 0}$ and $\D^{\leq0} := \D^{< 0} \D^0$ are Hopf subalgebras of $\D$.
\end{enumerate}

\bigbreak
\noindent $\bullet$  The algebra $\D$ admits an exhaustive ascending filtration that satisfies $\gr \D \simeq \ku[X_1, \dots, X_5, T^{\pm1}]$. Hence
$\D$ is a noetherian domain.

\bigbreak
\noindent $\bullet$
The subalgebra $\Nuc\coloneqq \ku\langle x,u,g^{\pm 1}\rangle$ is a commutative Hopf subalgebra, hence $\Nuc \simeq \cO(\Gb)$, 
where $\Gb$ is the algebraic group as in the Introduction.  
Let $e, f, h$  be  the Chevalley generators of  $\spl_2(\ku)$, i.~e.
$[e,f]=h$,  $[h,e]=2e$, $[h,f]=-2f$.
The Hopf algebra map $\pi \colon \D \to U(\spl_2(\ku))$ determined by
\begin{align}\label{eq:iso-hopf}
\pi (v)&= \tfrac{1}{4} e,& \pi(y) &= 2f  ,& \pi (\zeta)&=-\tfrac{1}{2} h,&
\pi (u)&= \pi(y) =  \pi (g-1) =0,
\end{align}
induces  an
isomorphism of Hopf algebras $\D/\D\Nuc^+ \simeq U(\spl_2(\ku))$.   

\begin{remark}
The Hopf algebra $\D$ is pointed with coradical $\D_0 =\ku\langle g^{\pm 1}\rangle$. Indeed, by \cite[5.3.4]{montgomery} 
it suffices to show that $\D$ admits an exhaustive coalgebra filtration $\D = \cup_{n\in \N_0} \D_{[n]}$ with 
$\D_{[0]} = \ku\langle g^{\pm 1}\rangle$. Let $\D_{[n]}$ defined recursively by
\begin{align*}
\D_{[0]} &= \ku\langle g^{\pm 1}\rangle, &
\D_{[1]} &= \D_{[0]} + \ku\{x,y,\zeta,u\},\\
\D_{[2]} &= \D_{[1]} + \ku\{v\}, &
\D_{[n]} &= \D_{[2]} \D_{[n-1]}, \qquad n\geq 3.  
\end{align*}
Using the PBW basis one can check this is an exhaustive coalgebra filtration.
\end{remark}

\section{Simple finite-dimensional  modules}\label{sec:simples-D}
\subsection{Overview}\label{subsec:strategy-simples-D}
Let $A$ be an algebra and $B$ a subalgebra. Let $\lmod{A}$
 (respectively $\IRR A$, $\Irr A$)  denote  the category of left $A$-modules
(respectively, the set of isomorphism classes of  simple objects  in $\lmod{A}$, the finite-dimensional ones).
Often we do not  distinguish a class in $\IRR A$ and one of its representatives. 
Let $\Ind_{B}^{A}: \lmod{B} \to \lmod{A}$ and $\Res^{B}_A: \lmod{A} \to \lmod{B}$ denote the induction  and restriction functors,
e.g. $\Ind_{B}^{A}(M) =  A \otimes_B M$.   Given $S\in \Irr A$, there exists $T\in\Irr B$ such that $T \leq \Res^B_A S$. 
By  the standard adjunction
\begin{align}\label{eq:adjunction-Ind-Res}
\Hom_A(\Ind^A_B M, N) &\simeq \Hom_B(M, \Res^B_A N), & N\in \lmod{A}, M\in \lmod{B},
\end{align}
$S$ is a quotient of $\Ind^A_B  T$.
 We apply this  (classical) remark twice to compute $\Irr \D$.
 First  we compute $\Irr \D^{\geq 0}$ by determining all simple quotients of $\Ind^{\D^{\geq0}}_H W$ for each $W \in \Irr H$, 
where $H \coloneqq \ku\langle g^{\pm 1},u,v\rangle$, cf. \cite{abdf-super}.
Then  we compute  $\Irr \D$ from  $\Irr \D^{\geq0}$ in the same way.

\subsection{Determination of \texorpdfstring{$\Irr \D^{\geq 0}$} {}}\label{subsec:Dgeq0-simples}
For each $(a,b) \in \ku^\times \times\ku$ there is a one-dimensional $H$-module, denoted by $\ku_{a,b}$, with basis $\tx_{a,b}$ and action
\begin{align}\label{eq:action-k-a-b}
g\cdot \tx_{a,b} &= a \tx_{a,b}, & v\cdot \tx_{a,b} &= b \tx_{a,b}, &  u\cdot\tx_{a,b} &= 0.
\end{align}
Then $\Irr H = \{\ku_{a,b} \colon (a,b) \in \ku^\times \times\ku \}$ \cite[3.3]{abdf-super}. Let $W_{a,b}\coloneqq \Ind^{\D^{\geq0}}_H\ku_{a,b}$.

\begin{lemma}
The elements $\tx^{(n)}_{a,b} \coloneqq\zeta^{n} \cdot \tx_{a,b}$, $n\geq 0$, form a basis of $W_{a,b}$.
\end{lemma}
\pf  Indeed,
$\D^{\geq0} \ot_H \ku_{a,b}\simeq (U(\kh)\ot_\ku H)\ot_H \ku_{a,b} 
\simeq U(\kh) \ot_\ku\ku_{a,b}$.
\epf

\begin{lemma}\label{lemma:Wab-is-simple-if-b-neq-0}
If $b\neq0$, then $W_{a,b}$ is simple for any $a\in \ku^\times$ .
\end{lemma}
\pf 
We use that $(v-b)^n \zeta= \zeta (v-b)^n+ n v (v-b)^{n-1} $ for all  $n\geq 0$, which is straightforward.
Notice that $(v-b)^n \cdot  \tx^{(n)}_{a,b}=b^n n!   \tx^{(0)}_{a,b}$ for all  $n\geq 0$.
Indeed, 
\begin{align*}
(v-b)^{n+1} \cdot  \tx^{(n+1)}_{a,b}=& (v-b) (v-b)^n \zeta\cdot  \tx^{(n)}_{a,b} \\
=& (v-b) \left( \zeta (v-b)^n+ n v (v-b)^{n-1} \right)\cdot  \tx^{(n)}_{a,b} \\
=&b^n n! (v-b)\zeta \cdot \tx^{(0)}_{a,b}+ n b^n n! v \cdot   \tx^{(0)}_{a,b}
=  b^{n+1} (n+1)! \tx^{(0)}_{a,b}.
\end{align*}
Let  $0 \neq z \in W_{a, b}$ and write $z =\sum_{k=0}^n c_k \tx^{(k)}_{a,b} $ with $n\geq 0$ and $c_n\ne 0$. 
Then  $(v-b)^n \cdot z = b^n n! c_n  \tx^{(0)}_{a,b}$. Thus $\langle z\rangle =W_{a,b}$ and the claim follows.
\epf

We next study the simple quotients of $W_{a,0}$, $a\in \kut$. Let
\begin{align*}
V_{a,c} &= \langle \tx^{(1)}_{a,0} + \tfrac{1}{2}c  \tx^{(0)}_{a,0} \rangle \leq  W_{a,0}, & T_{a,c}&\coloneqq W_{a,0}/V_{a,c}\,,&
c\in \ku.
\end{align*}
The choice of the coefficient $\frac{1}{2 }$ is convenient for calculations with $K_c$, see Remark \ref{rem:highest-weigth}.
Then $T_{a,c}$ is one-dimensional with basis $z_{a,c}$ and action
\begin{align}\label{eq:action-Lac}
g\cdot z_{a,c} &= a z_{a,c}, & \zeta\cdot z_{a,c} &= -\frac{1}{2}  c \ z_{a,c},& u\cdot z_{a,c} &= 0,&  v\cdot z_{a,c} &= 0.
\end{align}

\begin{lemma}\label{lemma:maximal-submodules-Wa0}
The set of maximal submodules of $W_{a,0}$ is $\{V_{a,c} \colon c \in \ku\}$.
\end{lemma}
\pf
The subcategory of $\D^{\geq0}$-modules where $u$, $v$ and $g-a$ act by $0$ is equivalent to the category of modules 
over $\D^{\geq0}/(u, v, g-a)\simeq \ku[\zeta]$. Now $W_{a,0}$ belongs to this subcategory, because  the action in the basis $\left\{\tx^{(n)}_{a,b}\right\}$ is
\begin{align*}
g\cdot \tx^{(n)}_{a,b}  &= g \zeta^n\cdot \tx_{a,b} = \zeta^n g\cdot \tx_{a,b} = a \tx^{(n)}_{a,b},\\
u \cdot \tx^{(n)}_{a,b} &= u\zeta^n \cdot \tx_{a,b} = \sum_{k=0}^n \binom{n}{k}\zeta^k u \cdot\tx_{a,b} = 0,\\
v \cdot \tx^{(n)}_{a,b} &= v\zeta^n \cdot \tx_{a,b} = \sum_{k=0}^n \binom{n}{k}\zeta^k v \cdot\tx_{a,b} = 0.   
\end{align*}
Under this correspondence, $W_{a,0}$ goes to the regular $\ku[\zeta]$-module; the claim follows.
\epf

\begin{prop}\label{prop:Dgeq0-simples}
$\Irr \D^{\geq0}\simeq \{T_{a,c}: (a,c) \in \ku^\times \times \ku\}$.
\end{prop}
\pf
Let $T  \in \Irr \D^{\geq0}$. Then there exists $(a,b)\in \ku^\times\times\ku$ such that $\ku_{a,b}$ is isomorphic to a submodule of $\Res^H_{\D^{\geq0}} T$. By \eqref{eq:adjunction-Ind-Res} $T$ is a quotient of $W_{a,b}$. If $b\neq0$, then $T \simeq W_{a,b}$ by Lemma \ref{lemma:Wab-is-simple-if-b-neq-0},  contradicting $\dim T < \infty$. Hence $b=0$ and $T\simeq T_{a,c}$ for some $c\in\ku$.
\epf

\subsection{Calculation of \texorpdfstring{$\Irr \D$}{}}\label{subsec:D-simples} 
The \emph{Verma module} $\verma_{a,c}$, $(a,c) \in \ku^\times \times \ku$, is 
\begin{align*}
\verma_{a,c} \coloneqq  \Ind^{\D}_{\D^{\geq0}}T_{a,c} = \D\ot_{\D^{\geq0}} T_{a,c}.
\end{align*}
\begin{lemma}
The elements $z^{(i,j)}_{a,c} \coloneqq y^i x^j \cdot z_{a,c}$, $i,j\geq 0$, form a basis of $\verma_{a,c}$.
\end{lemma}
\pf
Indeed, 
$\D \ot_{\D^{\geq0}} T_{a,c}\simeq (\D^{<0}\ot_\ku \D^{\geq0})\ot_{\D^{\geq0}} T_{a,c}\simeq \D^{<0}\ot_\ku  T_{a,c}$.
\epf

For the next proof we need the following formulas from \cite[Lemma 2.5]{ap}:
\begin{align}
\label{eq:comm-u-yn}
u y^n &= y^n u + n y^{n-1} - \sum_{k=0}^{n-1} \binom{n}{k+1} \frac{(k+1)!}{2^k} y^{n-1-k} x^k g, & n\geq1; \\
\label{eq:comm-v-xm}
v x^m &= x^m v + m x^{m-1}(1-g) + m x^m u, & m\geq1.
\end{align}

\begin{lemma}\label{lemma:Mac-simple-if-a-neq-1}
If $a\neq 1$ then $\verma_{a,c}$ is simple.
\end{lemma}
\begin{proof}
From \eqref{eq:comm-u-yn} we get the following formula for $i\geq 1$ and $ j\geq 0$:
\begin{align*}
u \cdot z_{a,c}^{(i,j)}& = (1-a) i z_{a,c}^{(i-1,j)} - a \sum_{k=1}^{i-1} \binom{i}{k+1} \frac{(k+1)!}{2^k}z_{a,c}^{(i-1-k,j+k)};
\end{align*}
 clearly $u \cdot z_{a,c}^{(0,j)} = 0$ for all $ j\geq 0$. Next we prove by induction that
\begin{align*}
u^i \cdot z_{a,c}^{(i,j)}& = (1-a)^i i! z_{a,c}^{(0,j)}, & i, j \geq 0. 
\end{align*}
Indeed, 
\begin{align*}
u^{i+1} \cdot z_{a,c}^{(i+1,j)}=& (1-a) (i+1) u^iz_{a,c}^{(i,j)} - a \sum_{k=1}^{i} \binom{i+1}{k+1} \frac{(k+1)!}{2^k}u^iz_{a,c}^{(i-k,j+k)}\\
=& (1-a)^{i+1} (i+1)! z_{a,c}^{(0,j)} \\
&- a \sum_{k=1}^{i} \binom{i+1}{k+1} \frac{(k+1)!}{2^k} (i-k)! (1-a)^{i-k} u^k z_{a,c}^{(0,j+k)} \\
=& (1-a)^{i+1} (i+1)! z_{a,c}^{(0,j)}.
\end{align*}
Thus $u^n \cdot z_{a,c}^{(i,j)} = 0$ if $n>i$. From \eqref{eq:comm-v-xm} we get  $v \cdot z_{a,c}^{(0,j)} = j (1-a) z_{a,c}^{(0,j-1)}$,
and it becomes evident that
\begin{align*}
v^j u^i \cdot z_{a,c}^{(i,j)} = (1-a)^{i+j} i! j! z_{a,c}^{(0,0)}.
\end{align*}
Then $v^m u^i \cdot z_{a,c}^{(i,j)} = 0$ for $m>j$ and $i\geq 0$. 
Given  $z \in \verma_{a,c}$, $z\neq 0$, write  
\begin{align*}
z &= \sum_{i=0}^{N} \sum_{j=0}^{M_i} d_{i,j} z_{a,c}^{(i,j)} \neq 0  \hspace{20pt}  \text{ with } d_{N,M_N} \neq 0,
\\
\text{hence }\hspace{20pt}  v^{M_N} u^{N} \cdot z&= d_{N, M_N} (1-a)^{M_N + N} N! M_N! z_{a,c}^{(0,0)}.
\end{align*}
Since $a \neq 1$ and $\verma_{a,c} =\langle z_{a,c}^{(0,0)}\rangle$ the Lemma follows. 
\end{proof}

\begin{coro}\label{coro:every-finite-D-simple-is-quotient}
Every $S\in \Irr \D$  is a quotient of $\verma_{1,c}$ for some $c\in\ku$. \qed
\end{coro}

Thus we need to study the Verma modules $\verma_{1,c}$. For the next lemma we use the following commutation relation from \cite[Lemma 2.5]{ap}:
\begin{align}\label{eq:comm-y-g}
g^n y^\ell &= \sum_{k=0}^{\ell} \binom{\ell}{k} \frac{[2n]^{[k]}}{2^k}  y^{\ell-k} x^k g^n, && n,\ell\in\N_0.
\end{align}

\begin{lemma}\label{lemma:action-g-1-locally-nilpotent}
The action of $g-1$ in $\verma_{1,c}$ is locally nilpotent. 
\end{lemma}
\begin{proof}
We prove recursively on $i$  that for each $i,j \geq 0$ there exists $n_{i,j}\in \N$ such that $(g-1)^{n_{i,j}} \cdot z_{1,c}^{(i,j)} = 0$. 
If $i=0$, then $g \cdot z_{1,c}^{(0,j)} = z_{1,c}^{(0,j)}$ for  $j\in \N$, hence we  take $n_{0,j} = 1$. 
Given such $n_{k,j}$ for every $j\in\N$ and $k<i$, we have
\begin{align*}
(g-1)\cdot z_{1,c}^{(i,j)} &= (g-1)y^i z_{1,c}^{(0,j)} =  \sum_{k=1}^{i} \binom{i}{k} \frac{[2]^{[k]}}{2^k} z_{1,c}^{(i-k,j+k)}
\end{align*}
by \eqref{eq:comm-y-g}. Taking $n_{i,j} = \max_{0\leq k < i}\{n_{k,i+j-k}\} + 1$,  the Lemma follows.
\end{proof}
\begin{prop}\label{prop:g-1-x-u-D-nilpotent} Let $M\in \lmod{\D}$, $\dim M < \infty$, with associated representation  $\rho\colon \D\to \End M$.
Then  $g-1$, $x$ and $u$ act  nilpotently on $M$.
\end{prop}

\begin{proof} Arguing by induction on $\dim M$, we may assume that
$M\in \Irr \D$.  Then  $g-1$ acts nilpotently on $M$ by Corollary \ref{coro:every-finite-D-simple-is-quotient} and Lemma \ref{lemma:action-g-1-locally-nilpotent}.
Recall that a linear operator $T$ on a finite-dimensional space is nilpotent if and only if $\Tr(T^n) = 0$ for every $n\in \N$. 
Since $x^{n+1} = \frac{2}{n}\left(x^n y - y x^n\right)$ for  $n\in\N$ we get
$\Tr(\rho(x)^{n+1}) = \frac{2}{n}\Tr(\rho(x)^n \rho(y) - \rho(y) \rho(x)^n) = 0$.
Since $x = \zeta x - x \zeta$ we also have that $\Tr(\rho(x)) = 0$. So $\rho(x)$ is nilpotent. The  argument for $u$ 
is similar using the relations $u^{n+1} = \frac{2}{n}\left(u^n v - v u^n\right)$ and $u = u \zeta -  \zeta u$. 
\end{proof}
We  now determine $\Irr \D$ via an argument  connecting with \cite[1.11]{ap}. 

\begin{lemma}\label{lema:quotient-nilpotent}
Let $A$ be an algebra and $ \cF  \subset A$  a family of   elements satisfying
\begin{enumerate}[leftmargin=*,label=\rm{(\alph*)}]
\item\label{item:quotient-nilpotent-commuting} the elements of $\cF$  commute with each other;

\item\label{item:quotient-nilpotent} for any $M\in \lmod{A}$, $\dim M < \infty$, any $x \in\cF$ acts nilpotently on $M$;

\item\label{item:quotient-nilpotent-normal} $\cF$ is normal, i.~e.  the vector subspace $I$  of $A$ generated by $\cF$ satisfies
\begin{align}\label{eq:left-twosided}
A I =I A.
\end{align}
\end{enumerate}

Let  $S \in \Irr A$. Then the representation $\rho: A\to \End S$ factorizes  through $A /  IA$.
Thus the projection $A \to A/IA$ induces a bijection
\begin{align*}
\Irr A \simeq \Irr A/IA.
\end{align*}
\end{lemma}

\pf  Let $\widetilde{I} = \rho(I)$ and $\cI = \rho(AI) =  \rho(IA) $. By  \ref{item:quotient-nilpotent-commuting} and \ref{item:quotient-nilpotent} 
there exists $r\in \N$  such that $\widetilde{I}^r = 0$. Then
$\cI^r = 0$ by \eqref{eq:left-twosided}. Hence the  ideal $\cI$ is contained in the Jacobson radical of $\rho(A)$. 
But by Burnside's Theorem \cite[(3.3.2)]{CR} $\rho(A) = \End S$ since $S$ is simple. So $IA$ acts by $0$ on $S$.
\epf


Recall the map $\pi: \D \to U(\spl_2(\ku))$ from \eqref{eq:iso-hopf}.

\begin{theorem}\label{thm:simple-D}
The map $\pi$ induces a bijection
$\Irr \D \simeq \Irr U(\spl_2(\ku))$.
\end{theorem}
\begin{proof}
Let $\cF = \{x, u, g-1\}$. By the defining relations of $\D$,
$\cF$ satisfies \ref{item:quotient-nilpotent-commuting} and \ref{item:quotient-nilpotent-normal}.  
By Proposition \ref{prop:g-1-x-u-D-nilpotent} $\cF$ satisfies \ref{item:quotient-nilpotent}. Thus 
Lemma \ref{lema:quotient-nilpotent} applies.
\end{proof}

\begin{remark}\label{rem:highest-weigth}  
For $n\in \N_0$, let $L_n\in \Irr \D$ correspond to the simple $\spl_2(\ku)$-module of highest weight $n$. 
Then $L_n$ has a basis $t_0,\dots,t_n$ where  the action is given by
\begin{align}\label{eq:D-action-Vc}
\begin{aligned}
y \cdot t_i &= t_{i+1}, & v \cdot t_i &= \frac{i}{2} (n-i+1) t_{i-1}, & \zeta\cdot t_i&= -\frac{1}{2}(n-2i) t_i,\\
x \cdot t_i &= 0, & u \cdot t_i &= 0, & g\cdot t_i &= t_i.
\end{aligned}
\end{align}
It can be shown that $L_n$ can be presented as quotient of the Verma module $\verma_{1,n}$. Indeed
let $\mathtt{M}_n$ be the Verma module over $\spl_2(\ku)$ of highest weight $n$. 
Then $\mathtt{M}_n \simeq K_n\coloneqq \verma_{1,n}/\widetilde{M}_n$ where 
$\widetilde{M}_n\coloneqq x \verma_{1,n} = \langle z_{1,n}^{(0,1)} \rangle \leq \verma_{1,n}$.
\end{remark}

\begin{coro}
The Hopf algebra $\D$ is spherical with pivot $g^{-1}$.
\end{coro}
\begin{proof}
By direct calculation in the generators we see that $\Ss^2(h) = g^{-1} h g$ for every $h\in\D$. It remains to show that for every $V\in\lmod{\D}$ with $\dim V < \infty$,  $\Tr_V(fg) = \Tr_V(fg^{-1})$ for every $f\in\End_\D(V)$. By \cite[Prop. 2.1]{aagtv} we only need to consider $V\in\Irr \D$. Since $\End_\D(V)\simeq \ku$, and $\Tr_V(g) = \Tr_V(g^{-1}) = \dim V$ by \ref{eq:D-action-Vc}, the claim follows.
\end{proof}

\section{A localization  of the double of the Jordan plane}\label{sec:localization}
\subsection{Weyl algebras and iterated Ore extensions}\label{subsec:localization-weyl}
We refer to \cite{MCR} for the notations and basic notions used here.
Let $R$ be a  commutative ring. Recall that  the Weyl algebra $A_1(R)$
is the $R$-algebra generated  by $p$ and $q$  satisfying $pq-qp=1$. Alternatively
it can be described as the Ore extension $A_1(R) \simeq R[q][p\,;\,\partial_q]$.
We shall also consider the algebra $A'_1(R)=R[q^{\pm 1}][p\,;\,\partial_q]$, which is the localization of $A_1(R)$ with respect
of the multiplicative set generated by $q$. Observing that $(qp)q-q(qp)=q$, we have an alternative description
of $A'_1(R)$ as a Laurent extension:
\begin{equation}
A'_1(R)=R[q^{\pm 1}][p\,;\,\partial_q]=R[qp][q^{\pm 1}\,;\,\sigma^{\pm 1}] 
\end{equation}with $\sigma$ the $R$-automorphism of $R[qp]$ defined by $\sigma(qp)=qp-1$.

The Weyl algebras $A_n(R)$ and their localizations are defined similarly for  $n \in \N$; then  
$A_n(R)=R[q_1,\ldots,q_n][p_1\,;\,\partial_{q_1}]\ldots[p_1\,;\,\partial_{q_1}]$ and 
\begin{align}\label{eq:weyl} 
\begin{aligned}
A'_n(R)&=R[q_1^{\pm 1},\ldots,q_n^{\pm 1}][p_1\,;\,\partial_{q_1}]\ldots[p_n\,;\,\partial_{q_n}]\\
&=R[q_1p_1,\,\ldots,q_np_n][q_1^{\pm 1}\,;\,\sigma_1^{\pm 1}]\ldots[q_n^{\pm 1}\,;\,\sigma_n^{\pm 1}].
\end{aligned}
\end{align}


The proof of the following Lemma is straightforward.

\begin{lemma}\label{lema:Dore} The algebra $\D$ can be described as an iterated Ore extension:
\begin{equation}\label{eq:Dore}
\D \simeq \underbrace{\ku[g^{\pm 1},x,u]}_{\Nuc \text{ commutative}}[y\,;\,d][\zeta\,;\,\delta][v\,;\,\sigma,\deriv] 
\end{equation}
where $d$ is the derivation of $\Nuc:=\ku[g^{\pm 1},x,u]$, $\delta$ is the derivation of $\Nuc[y\,;\,d]$,
$\sigma$ is the automorphism and $\deriv$ is the $\sigma$-derivation of $\Nuc[y\,;\,d][\zeta\,;\,\delta]$
defined by:
\begin{equation}\label{eq:valuesD}\begin{matrix}
d(x) = -\frac12x^2, & \delta(x) = x,\hfill & \sigma(x) = x,\hfill & \deriv(x)=1-g + xu,\hfill \\
d(u)=g-1,\hfill&\delta(u)=-u,\hfill &\sigma(u)=u,\hfill &\deriv(u)= -\frac12u^2,\hfill \\
d(g)=-xg,\hfill&\delta(g)=0,\hfill &\sigma(g)=g,\hfill &\deriv(g)= gu,\hfill \\
&\delta(y)=y,\hfill &\sigma(y)=y,\hfill & \deriv(y)=-g\zeta + yu,\hfill\\
&             &\sigma(\zeta)=\zeta+1,\hfill &\deriv(\zeta)=0.  \qed\hfill\\
\end{matrix} \end{equation}
\end{lemma}

\begin{coro}\label{coro:algebraic-properties}
The algebra $\D$ is strongly noetherian, AS-regular and Cohen-Macaulay.
\end{coro}

\pf  $\D$  is strongly noetherian by \cite[Proposition 4.10]{ASZ}; AS-regular by \cite[Proposition 2]{AST} 
and Cohen-Macaulay by \cite[Lemma 5.3]{ZZ}.
\epf

\subsection{Localizing}\label{subsec:localization}
We consider the following elements of the  subalgebra $\Nuc$:
\begin{equation}\label{eq:defqz}
q:=ux + 2(1+g)\quad\text{and}\quad z:=q^2g^{-1}.
\end{equation}

\begin{lemma}\label{lema:xqf} 
\begin{enumerate}[label=\rm{(\roman*)}]
\item \label{item:adx-nilpot} $x$  is ad-locally nilpotent in $\D$;
\item \label{item:q-normal} $q$  is normal in $\D$,
\item\label{item:z-central} $z$  is central in $\D$.
\end{enumerate}\end{lemma}
\begin{proof}
\ref{item:adx-nilpot}: It is clear that  ${\rm ad}_x(x)={\rm ad}_x(u)={\rm ad}_x(g^{\pm 1})={\rm ad}_x^2(\zeta)=0$; then
${\rm ad}_x^2(y)={\rm ad}_x(-\frac12x^2)=0$ and ${\rm ad}_x^2(v)={\rm ad}_x(-xu + g - 1)=0$.
By the Leibniz rule, the claim follows.

\ref{item:q-normal}: By straightforward calculations we have
\begin{equation}\label{qnormal}
qy=\textstyle(y+\frac12x)q,\quad qv=(v-\frac12u)q,\quad q\zeta=\zeta q.
\end{equation}
Since $q$ commmutes with $u,x,g^{\pm 1}$, we have $\D q = q\D$, i.~e. $q$ is normal.

\ref{item:z-central}: Clearly $g$ commutes with $x$, $u$ and $\zeta$, and  satisfies:
$gy=(y+x)g$, $gv=(v- u)g$. Hence $z =q^2g^{-1}$ is central in $\D$.
\end{proof}

Lemma \ref{lema:xqf} allows us to consider the localization $\D'$ of $\D$ with respect to the multiplicative
set generated by $x$ and $q$.  Let us introduce the element
\begin{equation}\label{eq:deft}
t:=qx^{-1}=u+2(1+g)x^{-1} \in \Nuc':=\D'\cap\Nuc.
\end{equation}
Then $\ku[g^{\pm 1},x^{\pm 1},u]=\ku[g^{\pm 1},x^{\pm 1},t]$. In $\Nuc'$, the element $t$ is invertible with $t^{-1}=q^{-1}x$ 
and the element $z$ is invertible with $z^{-1}=gq^{-2}$.
Then:\begin{equation}\label{C'}
\Nuc'=\ku[g^{\pm 1},x^{\pm 1},t^{\pm 1}]=\ku[g^{\pm 1},q^{\pm 1},t^{\pm 1}]=\ku[z^{\pm 1},q^{\pm 1},t^{\pm 1}].      
\end{equation}
We deduce the following description of $\D'$ as an iterated Ore extension:
\begin{equation}\label{D'ore}
\D'=\underbrace{\ku[z^{\pm 1},q^{\pm 1},t^{\pm 1}]}_{{\Nuc'} \text{ commutative}}[y\,;\,d][\zeta\,;\,\delta][v\,;\,\sigma, \deriv];
\end{equation}
here $d$, $\delta$, $\sigma$, $\deriv$ denote the canonical extensions to $\D'$ of $d$, $\delta$, $\sigma$, $\deriv$ as in  \eqref{eq:Dore}. 

\medbreak
Let $R:=\ku[z^{\pm 1},t^{\pm 1}]$. We introduce
\begin{equation}\label{eq:defp}
p:=-2q^{-2}ty.
\end{equation}

\begin{lemma}\label{step1}
The subalgebra  $\ku[z^{\pm 1},q^{\pm 1},t^{\pm 1}][y\,;\,d]=\ku[z^{\pm 1},q^{\pm 1},t^{\pm 1}][p\,;\,\partial_q]$ of $\D'$ 
is isomorphic to $A'_1(R)$.
\end{lemma}
\begin{proof} We have $d(z)=0$ by Lemma \ref{lema:xqf}(iii). We  compute using \eqref{eq:valuesD}:
\begin{align*}
d(q)&\textstyle =d(u)x+ud(x)+2d(g)=-gx-x-\frac12ux^2=-\frac12xq=-\frac12q^2t^{-1},\\
d(t)&\textstyle =d(q)x^{-1}-qx^{-2}d(x)=0.\end{align*}
Then the change of variable $p:=-2q^{-2}ty$ leads to the commutation relations
$pz-zp=pt-tp=0$ and $pq-qp=1$.\end{proof}

\begin{lemma}\label{step2} The following subalgebras of $\D'$ are equal:
\begin{equation*}
\ku[z^{\pm 1},q^{\pm 1},t^{\pm 1}][y\,;\,d][\zeta\,;\,\delta]=\ku[z^{\pm 1},q^{\pm 1},t^{\pm 1}][p\,;\,\partial_q][\zeta\,;\,-t\partial_t].
\end{equation*}
\end{lemma}
\begin{proof}  We have $\delta(z)=0$ by Lemma \ref{lema:xqf}(iii). We compute using \eqref{eq:valuesD}:
\begin{align*}
\delta(t) &= \delta(u)+2(1+g)\delta(x^{-1})=-u+2(1+g)(-x^{-1})=-t,\\
\delta(q) &= \delta(t)x+t\delta(x)=-tx+tx=0,\\
\delta(p) &=-2\delta(q^{-2})ty-2q^{-2}\delta(t)y-2q^{-2}t\delta(y)=2q^{-2}ty-2q^{-2}ty=0,\end{align*}which gives the desired result.
\end{proof}

\begin{remark}The change of variable $s:=-t^{-1}\zeta$ leads to the commutation relations
$sz-zs=sq-qs=sp-ps=0$ and $st-ts=1$. 
Then denoting $T:=\ku[z^{\pm 1}]$, we have in $\D'$ the equality of subalgebras:
\begin{equation*}
\ku[z^{\pm 1},q^{\pm 1},t^{\pm 1}][y\,;\,d][\zeta\,;\,\delta]=\ku[z^{\pm 1}][q^{\pm 1}][p\,;\,\partial_q][t^{\pm 1}][s\,;\,\partial_t]\simeq A'_2(T).
\end{equation*}\end{remark}

\begin{lemma}With the change of variable $w:=t^{-1}v$, we have:
\begin{equation*}
\D'=\ku[z^{\pm 1},q^{\pm 1},\zeta][p\,;\,\partial_q][w\,;\,\derv][t^{\pm 1};\,\tau^{\pm 1}],
\end{equation*}
where $\derv$ is the derivation of $\ku[z^{\pm 1},q^{\pm 1},\zeta][p\,;\,\partial_q]$
such that:\begin{align}\derv(z)&=\derv(\zeta)=0,\\
\derv(q)&=\textstyle-1+\frac12q-z^{-1}q^2,\label{deltaq}\\ \derv(p)&=\textstyle-\frac12p+2qz^{-1}p+2z^{-1}\zeta+2q^{-2}-2z^{-1}\label{deltap},
\end{align}
and $\tau$ is the automorphism of $\ku[z^{\pm 1},q^{\pm 1},\zeta][p\,;\,\partial_q][w\,;\,\derv]$ such that:
\begin{align}
\tau(z)&=z,\quad \tau(q)=q,\quad \tau(p)=p,\\ \tau(\zeta)&=\zeta+1,\quad \tau(w)=\textstyle w+\frac12-2qz^{-1}\label{tauzeta}.\end{align}\end{lemma}
\begin{proof}
We start with the description \eqref{D'ore} of $\D'$ and recall Lemma \ref{step2}. 
By direct calculations using \eqref{eq:valuesD}, we show that:
\begin{align*}
vz&=\textstyle zv,\quad v\zeta=(\zeta +1)v,\quad vt=tv -\frac12t^2+2qz^{-1}t^2, \\ 
vq&=\textstyle qv-t+\frac12tq-tz^{-1}q^2,\\
vp&=\textstyle pv-\frac12tp+2tz^{-1}qp+2tz^{-1}\zeta+2tq^{-2}-2tz^{-1}.\end{align*}
We replace in $\D'=\Nuc'[p\,;\,\partial_q][\zeta\,;\,-t\partial_t][v\,;\,\sigma, \deriv]$ the generator $v$ by:
\begin{equation}\label{defw}
w:=t^{-1}v.  
\end{equation}
The last two of the above relations become:
\begin{align*}
wq&=\textstyle qw-1+\frac12q-z^{-1}q^2,\\
wp&=\textstyle pw-\frac12p+2qz^{-1}p+2z^{-1}\zeta+2q^{-2}-2z^{-1}.\end{align*}
We still have $wz=zw$.
We deduce from relations $\zeta t=t\zeta-t$ and $\zeta v=v\zeta -v$ that $w\zeta=\zeta w$.
Finally the relation $wt=tw -\frac12t+2qz^{-1}t$ can be rewritten as $tw=(w+\frac12-2qz^{-1})t$, 
which gives rise to the desired description of $\D'$.\end{proof}

Next we introduce the element
\begin{align}
z'&:=\textstyle q^{-1}\left[xv+uy+(-\frac12ux+g-1)\zeta -2(1+g)\right]\label{eq:expz'}\\
&=\textstyle \left[xv+uy+(-\frac12ux+g-1)\zeta -2(1+g)\right]q^{-1}\label{eq:expz'bis}
\end{align}

\begin{theorem}\label{thm:D'}
The algebra $\D'$ is isomorphic to the localized Weyl algebra $A'_2(S)$, with center $S:=\ku[z^{\pm 1},z']$.
\end{theorem}

In particular, $z'$ is central in $\D'$.

\begin{proof}
Since the subalgebra  $\ku[z^{\pm 1},q^{\pm 1},\zeta][p\,;\,\partial_q]$ is isomorphic to the localized Weyl algebra $A'_1(S)$ for $S=\ku[z^{\pm 1},\zeta]$, it is natural
by  \cite[Lemma 4.6.8]{dix} to look for an element $f\in \ku[z^{\pm 1},q^{\pm 1},\zeta][p\,;\,\partial_q]$ such that $\derv$ is the inner derivation
$\ad_f$. By \eqref{deltaq} and \eqref{deltap}, such an element satifies:
\begin{align*}
fq-qf&=\textstyle -1+\frac12q-z^{-1}q^2,\\
fp-pf&=\textstyle (-\frac12+2z^{-1}q)p+2q^{-2}+2z^{-1}(\zeta-1).
\end{align*}
A solution is clearly:\begin{equation}\label{deff}f:=\textstyle -(1-\frac12q+z^{-1}q^2)p+2q^{-1}-2z^{-1}(\zeta-1)q.\end{equation}
Then we have by construction for any $h\in\ku[z^{\pm 1},q^{\pm 1},\zeta][p\,;\,\partial_q]$:
\begin{equation*}
(w-f)h=wh-fh=hw+\derv(h)-fh=hw+fh-hf-fh=h(w-f). 
\end{equation*}Moreover, we deduce from $\tau(\zeta)=\zeta+1$ that $\tau(f)=f-2z^{-1}q$. Then the second identity of \eqref{tauzeta} implies that $\tau(w-f)=w-f+\frac12$.
A first consequence is that the element:
\begin{equation}\label{defz'}
z':=\textstyle w-f-\frac12\zeta 
\end{equation}is central in $\D'$. We can replace the generator $w$ by $z'$ to obtain:
\begin{equation}
\D'=\ku[z^{\pm 1},q^{\pm 1},\zeta][p\,;\,\partial_q][z'][t^{\pm 1};\,\tau^{\pm 1}],
\end{equation}
where all generators pairewise commute except:
\begin{equation}
pq-qp=1\quad\text{and}\quad t\zeta-\zeta t=t.
\end{equation}
We can replace the generator $\zeta$ by:
\begin{equation}
\xi:=-t^{-1}\zeta 
\end{equation}to obtain the following differential description:
\begin{equation}\label{D'weyldiff}
\D'=\ku[z^{\pm 1},z'][q^{\pm 1}][p\,;\,\partial_q][t^{\pm 1}][\xi\,;\,\partial_t],
\end{equation}with 
\begin{equation}pq-qp=1 \quad \text{and} \quad \xi t-t\xi=1.\end{equation} 
We can alternatively replace the generator $p$ by:
\begin{equation}
r:=-qp 
\end{equation}to obtain the following automorphic description:
\begin{equation}\label{D'weylaut}
\D'=\ku[z^{\pm 1},z'][r][q^{\pm 1}\,;\,\sigma^{\pm 1}][\zeta][t^{\pm 1};\,\tau^{\pm 1}]  
\end{equation}with 
\begin{equation}qr=(r+1)q \quad \text{and} \quad t\zeta=(\zeta+1)t.\end{equation}
We conclude from \eqref{D'weyldiff} or \eqref{D'weylaut} that $\D'$ is isomorphic to the
localized Weyl algebra $A'_2(S)$ as in \eqref{eq:weyl} for $S=\ku[z^{\pm 1},z']$.
Since $\car \ku = 0$, $S = \mathcal Z \left(A'_2(S)\right)$.

The last step is to express the central element $z'$  according to the initial generators of $\D$.
Let us remind that $q=ux-2(1+g)$, $z=q^2g^{-1}$ and $p=-2q^{-1}x^{-1}y$. It follows that 
the expression $(1-\frac12q+z^{-1}q^2)p$ in formula \eqref{deff} is equal to $-q^{-1}uy$. Then:
$f=q^{-1}[-uy+2(1+g)-2g\zeta]$. 
Moreover $w=q^{-1}xv$ by \eqref{eq:deft} and \eqref{defw}, and we obtain obviously the relation \eqref{eq:expz'}.
The alternative expression \eqref{eq:expz'bis} follows then from \eqref{qnormal}
\end{proof}

\begin{remark} Observe in \eqref{eq:expz'}  that  $z'$ does not depend on negative powers of $x$; i.~e. $z$ lies 
in the localization of $\D$ by inverting only the powers of $q$.
\end{remark}

\subsection{The center of \texorpdfstring{$\D$}{}}\label{subsec:center}

Because of Theorem \ref{thm:D'}, it is natural to introduce
\begin{equation*}
\textstyle s:=xv+uy+(-\frac12ux+g-1)\zeta -2(1+g)\ \in \ \Nuc v\oplus\Nuc y\oplus \Nuc \zeta\oplus \Nuc,
\end{equation*}
which is normal in $\D$ with associated inner automorphism $\gamma_s=\gamma_q$.\medskip

Since $z=q^2g^{-1}$ is central in $\D$, we are lead to introduce:
\begin{align}
\theta:=s^2g^{-1}\in \mathcal Z(\D).\end{align}
Now $z'=q^{-1}s=sq^{-1}$ is central in $\D'$ by Theorem \ref{thm:D'}, hence 
\begin{align}\label{eq:def-omega}
\omega:=zz'\in \mathcal Z(\D).
\end{align} 
In other words, $\omega=qg^{-1}s$ is central in $\D$, with $\omega\in\Nuc v\oplus\Nuc y\oplus \Nuc \zeta\oplus \Nuc$.
The three elements $z,\theta,\omega$ are not algebraically independent, since 
\begin{equation}\label{eq:relation-central}
z\theta=\omega^2.
\end{equation}

\begin{theorem}\label{prop:center-D} The center of $\D$ is the commutative subalgebra generated by  $z$, $\omega$ and $\theta$, which is
isomorphic to the quotient $\ku[X,Y,Z]/(XZ-Y^2)$.
\end{theorem}
\begin{proof} Clearly,  $\mathcal Z(\D) = \mathcal Z(\D') \cap \D$.
Since $\mathcal Z(\D')=\ku[z^{\pm 1},z'] =\ku[z^{\pm 1},\omega]$, we need to determine $\ku[z^{\pm 1},\omega]\cap \D$.
Since $\ku[z,\omega] \subset \mathcal Z(\D)$, 
we have to consider the $\ku$-linear combinations of monomials $z^{-i}\omega^j$ for positive $i$.
For any integer $j\geq 0$, it follows from the relations in the iterated Ore extension \eqref{eq:Dore} that 
$\omega^j$ is of the form 
$\omega^j=(g^{-j}q^jx^j)v^j+\cdots$ where the rest if of degree $\leq j-1$ in $\D$. We deduce that a power $z^i$ with $i\geq 1$ divides $\omega^j$ in $\D$
if and only if $z^i$ divides $g^{-j}q^jx^j$ in $\Nuc$, that is if and only if $j\geq 2i$. Then a monomial $m=z^{-i}\omega^j$ with $i\geq 1,j\geq 0$ is in $\D$ if and only if
$j\geq 2i$ and we have in this case $m=\theta^i\omega^{j-2i}$. This is sufficient to complete the proof.\end{proof}

The (spectrum of)  $\ku[X,Y,Z]/(XZ-Y^2)$ is the well-known Kleinian surface of type $A_1$, i.~e.
the algebra of invariants of the polynomial ring $\ku[x_1,x_2]$ under the action of the involution $x_1\mapsto -x_1,x_2\mapsto -x_2$.

\subsection{The skew field of fractions of \texorpdfstring{$\D$}{}}\label{subsec:GK}

Let $R$ be a  commutative $\ku$-algebra which is a domain. With the notations of \S \ref{subsec:localization-weyl}, the algebra $A_n(R)$ admits a skew field of fractions $\Frac(A_n(R))=: D_n(K)$ where $K$ is the field of fractions of $R$. 
We denote in particular $\Ds_{n,s}(\ku)= D_n(K)$ when $K$ is a purely transcendental extension $\ku(z_1,\ldots,z_t)$ of degree $t$. Following the seminal paper \cite{gk}, we say that a noncommutative $\ku$-algebra $A$ which is a noetherian domain satisfies the Gelfand-Kirillov property when its skew field of fractions $\Frac A$ is $\ku$-isomorphic to a Weyl skew field $\Ds_{n,s}(\ku)$ for some integers $n\geq 1, s\geq 0$. 

It is obvious that the Jordan plane $J$ satisfies the Gelfand-Kirillov property since $\Frac J\simeq \Ds_{1,0}(\ku)=D_1(\ku)$. This is also the case for the bosonization algebras $\D^{< 0}\# \ku \Gamma$ and  $U(\mathfrak h) \# \D^{> 0}$ because we can prove by easy technical  calculations that  $\Frac(\D^{< 0}\# \ku \Gamma)\simeq\Frac(U(\mathfrak h) \# \D^{> 0})\simeq\mathcal \Ds_{1,1}(\ku)$. Finally we deduce from the previous study that the algebra $\D$ itself satisfies the Gelfand-Kirillov property.

\begin{coro}\label{coro:GK} 
The skew field of fractions of $\D$ is $\ku$-isomorphic to the Weyl skew field $\Ds_{2,2}(\ku)$.
\end{coro}
\begin{proof} 
By Theorem \ref{thm:D'}, 
we have $\Frac\D=\Frac A'_2(S)=\Frac A_2(S)=D_2(K)$ of center $K=\ku(z,z')$.\end{proof}

\end{document}